\documentclass[reqno]{amsart}
\usepackage{amsmath}
\usepackage[dvips]{graphicx}
\usepackage{amsfonts}
\usepackage{amssymb}
\usepackage{latexsym}
\graphicspath{{Img/}}
\newtheorem{theorem}{Theorem}
\theoremstyle{plain}

\newtheorem{lemma}{Lemma}

\newtheorem{problem}{Problem}

\newtheorem{remark}{Remark}

\numberwithin{equation}{section}

\begin{document}

\title[The Fermat-Torricelli point for tetrahedra]{A fundamental property of the Fermat-Torricelli point for tetrahedra in the three dimensional Euclidean Space}
\author{Anastasios N. Zachos}
\address{University of Patras, Department of Mathematics, GR-26500 Rion, Greece}
\email{azachos@gmail.com}

\begin{abstract}
We prove the following fundamental property for the Fermat-Torricelli point for four non-collinear and non-coplanar points forming a tetrahedron in $\mathbb{R}^{3},$ which states that:
The three bisecting lines having as a common vertex the Fermat-Torricelli point formed by each pair of equal angles, which are seen by the opposite edges of the tetrahedron  meet perpendicularly at the Fermat-Torricelli point. Furthermore, we give an alternative proof, which is different from the one obtained by Bajaj and Mehlhos for the unsolvability of the Fermat-Torricelli problem for tetrahedra in $\mathbb{R}^{3}$ using only algebraic computations for some angles, which have as a common vertex the Fermat-Torricelli point of the tetrahedron.
\end{abstract}\maketitle

\section{Introduction}

The Fermat Problem for four non-collinear and non-coplanar points $A_{i}(x_{i},y_{i},z_{i})$ forming a tetrahedron $A_{1}A_{2}A_{3}A_{4}$  in $\mathbb{R}^{3}$ states that (\cite{abuSaymehHajja:97},\cite{KupitzMartini:94}, \cite{KupitzMartini:99}, \cite{Zach/Zou:09}):

\begin{problem}[The Fermat problem for $A_{1}A_{2}A_{3} A_{4}$ in $\mathbb{R}^{3}$]\label{WFN}
Find $A_{0}(x,y,z)$ in $\mathbb{R}^{3},$ such that:

\begin{equation}\label{objectivewfrn}
f(\{A_{0}\})=\sum_{i=1}^{4}\sqrt{(x-x_{i})^2+(y-y_{i})^2+(z-z_{i})^2}\to min.
\end{equation}

\end{problem}

The unsolvability of the Fermat-Torricelli problem for tetrahedra in $\mathbb{R}^{3}$ has been proved by Bajaj, Mehlhos, Melzak and Cockane in \cite{Bajaj:88},\cite{Mehlhos:00}, \cite{MelzakCockayne:69}, by applying Galois theory in some specific examples. Therefore, there is no Euclidean construction to locate the Fermat-Torricelli point $A_{0}$ of $A_{1}A_{2}A_{3}A_{4}$ in $\mathbb{R}^{3}.$

It is worth mentioning that Synge (\cite{Synge:87}) was the first who gave a non-Euclidean construction for the Fermat-Torricelli point using
some spindles. He considered around the two opposite edges of the tetrahedron and created two isosceles triangles containing two angles $\pi -\alpha_{102}$  and $\pi -\alpha_{304}$ and by rotating two circular arcs having as chords these skew edges he showed that there is a common value such that $\alpha_{102}=\alpha_{304},$ whuch yields a unique touching point (Fermat-Torricelli point) of the two spindles. Rubinstein, Thomas and Weng (\cite{RubinsteinThomasWeng:02}) use a more specific construction to find a Steiner tree having two Fermat-Torricelli points for a tetrahedron in $\mathbb{R}^{3},$ which is based on the Simpson line formed by the two vertices of two equilateral triangles with side lengths the two skew edges and located at the exterior of the tetrahedron. Recently, we use a similar construction with Synge (\cite[Theorem~5]{Zachos:21}) by constructing some isosceles triangles at the exterior of the skew edges of the tetrahedron, in order locate the Fermat-Torricelli point inside a tetrahedron in $\mathbb{R}^{3}.$

Kupitz, Martini, Abu-Saymeh, Hajja proved various properties of the Fermat-Torricelli point for some specific classes of tetrahedra having their opposite edges equal (isosceles tetrahedra) (\cite{KupitzMartini:94},\cite{abuSaymehHajja:97}).

It is well known that the existence and uniqueness of the Fermat point $A_{0}$ in $\mathbb{R}^{3}$ is derived by the convexity of the Euclidean norm (distance) and compactness arguments.

Sturm and Lindelof gave a complete characterization of the
solutions of the Fermat problem for $m$ given
points in $\mathbb{R}^{n}$ (\cite{Sturm:84},\cite{Lindelof:67}). Kupitz and Martini gave an alternative proof by using subdifferential calculus
(\cite{KupitzMartini:97}, \cite{KupitzMartini:99}). Eriksson and Noda Sakai Morimoto discovered some new characterizations for the Fermat-Torricelli point for tetrahedra in $\mathbb{R}^{3}$ ((\cite{Eriksson:97}, \cite{NodaSakaiMorimoto:91}).

We shall focus on the characterization of solutions for $m=4,$ $n=3$ ((\cite{KupitzMartini:97}, \cite{KupitzMartini:99}))
Let $\{A_{1},A_{2},A_{3},A_{4}\}$ be a tetrahedron and $A_{0}$ be a point in $\mathbb{R}^{3}.$ We denote by
$\vec {u}(A_j,A_i)$ the unit vector from $A_{j}$ to $A_{i}$ for $i,j=0,1,2,3,4.$

Two cases may occur:

(I) If for each point  $A_{i}\in \{A_{1},A_{2},A_{3},A_{4}\}$
\[ \|\sum_{j=1,j\ne i}^{4}\vec {u}(A_j,A_i)\|>1, \]
for $i,j=1,2,3,4,$ then

(a) $A_{0}$ does not belong to $\{A_{1},A_{2},A_{3},A_{4}\},$

(b) $ \sum_{i=1}^{4}\vec {u}(A_0,A_i)=\vec{0}$

(Fermat-Torricelli solution).

(II) If there is a point $A_{i}\in\{A_{1},A_{2},A_{3},A_{4}\}$
satisfying
\[ \|{\sum_{j=1,j\ne i}^{4}\vec {u}(A_j,A_i)}\|\le 1.
\] for $i,j=1,2,3,4,$ then  $A_{0} \equiv A_i$

(Fermat-Cavallieri solution).

Hence, we get two characterization of solutions for the Fermat problem for $A_{1}A_{2}A_{3}A_{4}$ in $\mathbb{R}^{3}.$

The Fermat-Torricelli solution is a tree, which consists of the quadruple of line segments $\{A_{0}A_{1},A_{0}A_{2},A_{0}A_{3},A_{0}A_{4}\}.$

The Fermat-Cavallieri tree solution is a tree, which consists of the triad of line segments $\{A_{i}A_{j},A_{i}A_{k},A_{i}A_{l}\},$
for $i,j,k,l=1,2,3,4, i\neq j\ne k\ne l.$

Abu-Abas, Abu-Saymeh and Hajja proved the non-isogonal property of the Fermat-Torricelli point in $\mathbb{R}^{3}$ (\cite{AbuAbbasHajja}, \cite{abuSaymehHajja:97})

In this paper, we prove a fundamental property of the Fermat-Torricelli point for tetrahedra in $\mathbb{R}^{3},$ by using basic algebra of vectors, which are transformed in spherical coordinates in $\mathbb{R}^{3}$ and we give an alternative proof for the unsolvability of the Fermat-Torricelli point for tetrahedra in $\mathbb{R}^{3},$ by obtaining an implicit expression for two angles having as a common vertex the Fermat-Torricelli point. Our main results are:  

Main Result 1.

The three bisecting lines having as a common vertex the Fermat-Torricelli point $A_{0}$ and formed by each pair of equal angles, which are seen by the opposite edges of $A_{1}A_{2}A_{3}A_{4}$  meet perpendicularly at $A_{0}$ (Section~2, Theorem~\ref{fundproperty})

Main Result 2.

The Fermat-Torricelli problem for four non-collinear and non-coplanar points forming a tetrahedron in $\mathbb{R}^{3}$ is not in general solvable by Euclidean constructions (Section~3, Theorem~\ref{unsolvabilityFTtetrahedron}).

\section{A fundamental property of the Fermat-Torricelli point for tetrahedra in $\mathbb{R}^{3}$}

%--------------------------------------------------------------------------------------------------------

Let $A_{1}A_{2}A_{3}A_{4}$ be a tetrahedron and $A_{0}$ be the Fermat-Torricelli point inside $A_{1}A_{2}A_{3}A_{4}$ in $\mathbb{R}^{3}$
We denote by $a_{i,j0k}$ the angle that is formed by the line segment that connects $A_{0}$ with the trace of the orthogonal projection $A_{i}$
to the plane defined by the triangle $\triangle A_{j}A_{0}A_{k}$  with the line segment $A_{i}A_{0}.$
We set $\alpha_{i0j}\equiv \angle A_{i}A_{0}A_{j}.$

We need the following well known lemma (\cite{Mehlhos:00},\cite{Synge:87}), in order to prove the main result (Theorem~\ref{fundproperty}):

\begin{lemma}\label{sumcosines}
If  \[ \|\sum_{j=1,j\ne i}^{4}\vec {u}(A_j,A_i)\|>1, \]
then
\begin{equation}\label{coseq1}
\cos\alpha_{102}=\cos\alpha_{304},
\end{equation}

\begin{equation}\label{coseq2}
\cos\alpha_{203}=\cos\alpha_{104},
\end{equation}

\begin{equation}\label{coseq3}
\cos\alpha_{103}=\cos\alpha_{204}
\end{equation}

and

\begin{equation}\label{coseq4}
1+\cos\alpha_{102}+\cos\alpha_{103}+\cos\alpha_{104}=0.
\end{equation}

\end{lemma}

\begin{proof}

The inner product of the unit vectors $\vec {u}(A_{0},A_{i}),$ $\vec {u}(A_{0},A_{j})$ yields:
\begin{equation}\label{innerproducti0j}
\vec {u}(A_{0},A_{i})\cdot \vec {u}(A_{0},A_{j})=\cos\alpha_{i0j},
\end{equation}

for i,j=1,2,3,4.

Taking into account the balancing condition of unit vectors $\vec {u}(A_{0},A_{i}),$ for $i=1,2,3,4,$ we get:

\begin{equation}\label{veccond1}
\vec {u}(A_{0},A_{1})+\vec {u}(A_{0},A_{2})= - (\vec {u}(A_{0},A_{3})+\vec {u}(A_{0},A_{4})),
\end{equation}

\begin{equation}\label{veccond2}
\vec {u}(A_{0},A_{2})+\vec {u}(A_{0},A_{3})= - (\vec {u}(A_{0},A_{1})+\vec {u}(A_{0},A_{4}))
\end{equation}

\begin{equation}\label{veccond3}
\vec {u}(A_{0},A_{1})+\vec {u}(A_{0},A_{3})= - (\vec {u}(A_{0},A_{2})+\vec {u}(A_{0},A_{4}))
\end{equation}

\begin{equation}\label{veccond4}
\vec {u}(A_{0},A_{4})= - \vec {u}(A_{0},A_{1})+\vec {u}(A_{0},A_{2})+\vec {u}(A_{0},A_{3})).
\end{equation}

By squaring both parts of (\ref{veccond1}), (\ref{veccond2}),(\ref{veccond3}) and taking into account (\ref{innerproducti0j}), we obtain
(\ref{coseq1}), (\ref{coseq2}) and (\ref{coseq3}), respectively.

By squaring both parts of (\ref{veccond4})and by substituting (\ref{innerproducti0j}) in the derived equation and taking into account (\ref{coseq1}), (\ref{coseq2}) and (\ref{coseq3}), we get (\ref{coseq4}).
Therefore, we derive that $\alpha_{102}=\alpha_{304},$ $\alpha_{203}=\alpha_{104}$ and $\alpha_{103}=\alpha_{204}.$

\end{proof}

\begin{theorem}\label{fundproperty}
The three bisecting lines having as a common vertex the Fermat-Torricelli point $A_{0}$ and formed by each pair of equal angles, which are seen by the opposite edges of $A_{1}A_{2}A_{3}A_{4}$  meet perpendicularly at $A_{0}.$
\end{theorem}

\begin{proof}

We express the unit vectors $\vec {u}(A_{0},A_{i})$ for $i=1,2,3,4,$ using spherical coordinates:

\begin{equation}\label{spherical1}
\vec {u}(A_{0},A_{1})=(1,0,0),
\end{equation}

\begin{equation}\label{spherical2}
\vec {u}(A_{0},A_{2})=(\cos\alpha_{102},\sin\alpha_{102},0),
\end{equation}

\begin{equation}\label{spherical3}
\vec {u}(A_{0},A_{3})=(\cos a_{3,102} \cos\omega_{3,102},\cos a_{3,102} \sin\omega_{3,102},\sin a_{3,102} ),
\end{equation}

\begin{equation}\label{spherical4}
\vec {u}(A_{0},A_{4})=(\cos a_{4,102} \cos\omega_{4,102},\cos a_{4,102} \sin\omega_{4,102},\sin a_{4,102} ).
\end{equation}

The unit vector $\vec{\delta}_{i0j}$ of the angle bisector that corresponds to the angle $\alpha_{i0j}$ is given by:

\begin{equation}\label{anglebisectori0j}
\vec{\delta}_{i0j}=\vec {u}(A_{0},A_{i})+\vec {u}(A_{0},A_{j})
\end{equation}

for $i,j=1,2,3,4, i\neq j.$

By replacing (\ref{spherical1}), (\ref{spherical2}), (\ref{spherical3}), (\ref{spherical4}) in (\ref{anglebisectori0j}), we get:

\begin{equation}\label{coordinatesdelta102}
\vec{\delta}_{102}=(1+\cos\alpha_{102},\sin\alpha_{102},0)
\end{equation}

\begin{equation}\label{coordinatesdelta103}
\vec{\delta}_{103}=(1+\cos a_{3,102} \cos\omega_{3,102},\cos a_{3,102} \sin\omega_{3,102},\sin a_{3,102} )
\end{equation}

\begin{equation}\label{coordinatesdelta104}
\vec{\delta}_{104}=(1+\cos a_{4,102} \cos\omega_{4,102},\cos a_{4,102} \sin\omega_{4,102},\sin a_{4,102} )
\end{equation}

\begin{equation}\label{coordinatesdelta203}
\vec{\delta}_{203}=(\cos\alpha_{102}+\cos a_{3,102} \cos\omega_{3,102},\sin\alpha_{102}+\cos a_{3,102} \sin\omega_{3,102},\sin a_{3,102} )
\end{equation}

\begin{equation}\label{coordinatesdelta204}
\vec{\delta}_{204}=(\cos\alpha_{102}+\cos a_{4,102} \cos\omega_{4,102},\sin\alpha_{102}+\cos a_{4,102} \sin\omega_{4,102},\sin a_{4,102} )
\end{equation}

\begin{eqnarray}\label{coordinatesdelta304}
\vec{\delta}_{304}=(\cos a_{3,102} \cos\omega_{3,102}+\cos a_{4,102} \cos\omega_{4,102},\cos a_{3,102} \sin\omega_{3,102}+ \nonumber\\ \cos a_{4,102}\sin\omega_{4,102},\sin a_{3,102}+\sin a_{4,102} )
\end{eqnarray}

Taking into account (\ref{coordinatesdelta102}), (\ref{coordinatesdelta103}), (\ref{coordinatesdelta104}), (\ref{coordinatesdelta203}), (\ref{coordinatesdelta204}), (\ref{coordinatesdelta304}), we obtain that:
\begin{equation}\label{innerprod102203}
\vec{\delta}_{102}\cdot \vec{\delta}_{203}=1+\cos\alpha_{102}+\cos\alpha_{103}+\cos\alpha_{203}.
\end{equation}

\begin{equation}\label{innerprod102103}
\vec{\delta}_{102}\cdot \vec{\delta}_{103}=1+\cos\alpha_{102}+\cos\alpha_{103}+\cos\alpha_{203}.
\end{equation}

By applying Lemma~\ref{sumcosines} in (\ref{innerprod102203}), (\ref{innerprod102103}), we derive that:

\[\vec{\delta}_{102}\cdot \vec{\delta}_{203}=\vec{\delta}_{102}\cdot \vec{\delta}_{103}=0,\]
which yields that $\vec{\delta}_{102}\perp \vec{\delta}_{203} \perp \vec{\delta}_{103}. $
Therefore, $\vec{\delta}_{102}, \vec{\delta}_{203}, \vec{\delta}_{103} $ is an orthonormal system of unit vectors.

Taking into account (\ref{coordinatesdelta102}), (\ref{coordinatesdelta103}), (\ref{coordinatesdelta104}), (\ref{coordinatesdelta203}), (\ref{coordinatesdelta204}), (\ref{coordinatesdelta304}), we obtain that:

\begin{equation}\label{innerprod102304}
\frac{\vec{\delta}_{102}}{|\vec{\delta}_{102}|}\cdot \frac{\vec{\delta}_{304}}{|\vec{\delta}_{304}|}=\frac{1}{\sqrt{2(1+\cos\alpha_{102})}}\frac{1}{\sqrt{2(1+\cos\alpha_{304})}}(-2(1+\cos\alpha_{102})).
\end{equation}

By replacing $\alpha_{304}=\alpha_{102}$ (Lemma~\ref{sumcosines}) in (\ref{innerprod102304}), we derive that:

\[\frac{\vec{\delta}_{102}}{|\vec{\delta}_{102}|}\cdot \frac{\vec{\delta}_{304}}{|\vec{\delta}_{304}|}=-1.\]
Hence, the angle bisectors of the angles $\alpha_{102}$ and $\alpha_{304}$ belong to the same line.

By following the same process and by applying Lemma~\ref{sumcosines}, we get:
\[\frac{\vec{\delta}_{203}}{|\vec{\delta}_{203}|}\cdot \frac{\vec{\delta}_{104}}{|\vec{\delta}_{104}|}=-1\]
and
\[\frac{\vec{\delta}_{103}}{|\vec{\delta}_{103}|}\cdot \frac{\vec{\delta}_{204}}{|\vec{\delta}_{204}|}=-1,\]
which yields that the angle bisectors of the angles $\alpha_{203}$ and $\alpha_{104}$ belong to the same line
and the angle bisectors of the angles $\alpha_{103}$ and $\alpha_{204}$ belong to the same line, respectively.
\end{proof}

\begin{remark}
We should not confuse the fundamental property of the Fermat-Torricelli tree solution for tetrahedra with the Steiner tree solution for tetrahedra having two nodes (Fermat-Torricelli points) in $\mathbb{R}^{3}.$ A specific case was proved in \cite[Theorem~3.1]{RubinsteinThomasWeng:02} for the Steiner tree problem, which states that:
if $A_{1}A_{2}A_{3}A_{4}$ is a centralized symmetric tetrahedron, then its three Simpson lines meet perpendicularly at its center $A_{0}.$ Therefore, the fundamental property of the Fermat-Torricelli tree for any tetrahedron $A_{1}A_{2}A_{3}A_{4},$  where three bisecting lines meet perpendicularly is much stronger than the one proved for Steiner trees for centralized symmetric tetrahedra in $\mathbb{R}^{3}.$  
\end{remark}

\section{Unsolvability of the Fermat-Torricelli problem for tetrahedra in $\mathbb{R}^{3}$ by Compass and Ruler}

We need the following lemma, which gives the position of four rays, which meet at a common point, in order to prove the unsolvability of the Fermat-Torricelli problem for $A_{1}A_{2}A_{3}A_{4}$ in $\mathbb{R}^{3}$ by compass and ruler. 

\begin{lemma}\cite[Proposition~1]{Zachos:20}\label{angle304}
The position of four line segments $A_{i}A_{0}$ which meet at $A_{0}$ for $i=1,2,3,4$ depend on exactly five given angles
$\alpha_{102},$ $\alpha_{103},$ $\alpha_{104},$ $\alpha_{203}$ and
$\alpha_{204}.$

The sixth angle $\alpha_{304}$ is calculated by the following formula:

\begin{eqnarray}\label{calcalpha3042}
&&\cos\alpha_{304}=\frac{1}{4} [4 \cos\alpha _{103} (\cos\alpha
_{104}-\cos\alpha _{102} \cos\alpha _{204})+\nonumber\\
&&+2 \left(b+2 \cos\alpha _{203} \left(-\cos\alpha _{102}
\cos\alpha _{104}+\cos\alpha _{204}\right)\right)] \csc{}^2\alpha
_{102}\nonumber\\
\end{eqnarray}

where

\begin{eqnarray}\label{calcalpha304auxvar}
b\equiv\sqrt{\prod_{i=3}^{4}\left(1+\cos\left(2 \alpha
_{102}\right)+\cos\left(2 \alpha _{10i}\right)+\cos\left(2 \alpha
_{20i}\right)-4 \cos\alpha _{102} \cos\alpha _{10i} \cos\alpha
_{20i}\right)}\nonumber\\.
\end{eqnarray}

for $i,k,m=1,2,3,4,$ and $i \ne k \ne m.$
\end{lemma}

\begin{theorem}\label{unsolvabilityFTtetrahedron}
The Fermat-Torricelli problem for four non-collinear and non-coplanar points forming a tetrahedron in $\mathbb{R}^{3}$ is not in general solvable by Euclidean constructions.
\end{theorem}

\begin{proof}
By replacing $\alpha_{304}=\alpha_{102},$ $\alpha_{104}=\alpha_{203}$ and $\alpha_{103}=\alpha_{204}$ in (\ref{calcalpha3042}) taken from Lemma~\ref{angle304}, we derive the implicit function $a_{102}=g(\alpha_{102},\alpha_{203}).$ Therefore, we cannot derive in general from this implicit expression an explicit function of $\alpha_{102}$ with respect to $\alpha_{203}.$ Hence, this functional dependence is responsible for the unsolvability of the Fermat-Torricelli problem for tetrahedra in $\mathbb{R}^{3}.$
\end{proof}

\section{Concluding Remarks}

By enriching Synge's construction with the fundamental property of the Fermat-Torricelli point for tetrahedra in $\mathbb{R}^{3},$ we may obtain the Fermat-Torricelli tree sausage, in order to develop some models for  the determination of 3-D minimum-energy configurations for macromolecular structures such as proteins and DNA, instead of working with the method of Steiner minimal trees established by Stanton and J. Mc Gregor Smith and W. Smith (\cite{StantonSmith}),\cite{SmithSMith:95}).

\end{document}